\documentclass[11pt]{amsart}

\usepackage{graphicx}
\usepackage{amsthm,amsmath,amssymb,tikz,bm}
\usepackage{mathtools}
\usepackage{xifthen}
\usepackage{comment}
\usepackage[margin=4.5cm]{geometry}
\usepackage[shortlabels]{enumitem}
\usepackage[colorlinks=true,
linkcolor=blue,citecolor=blue,
urlcolor=blue]{hyperref}

\newtheorem{theorem}{Theorem}[section]
\newtheorem{lemma}[theorem]{Lemma}
\newtheorem{sublemma}{}[theorem]
\newtheorem{corollary}[theorem]{Corollary}

\newtheorem{proposition}[theorem]{Proposition}

\newcommand{\F}{\mathbb F}
\newcommand{\GF}{\operatorname{GF}}
\newcommand{\PG}{\operatorname{PG}}
\newcommand{\AG}{\operatorname{AG}}
\newcommand{\BB}{\operatorname{BB}}

\newcommand{\spn}{\operatorname{span}}
\newcommand{\exq}{\operatorname{ex}_q}
\newcommand{\qexp}{\operatorname{Exp}_q}

\title[Matroidal K\H{o}vari--S\'os--Tur\'an theorem]{The K\H{o}vari--S\'os--Tur\'an theorem for $\GF(q)$-representable matroids}

\author{Wayne Ge}
\address{Mathematics Department \\
Louisiana State University \\ Baton Rouge, LA}
\email{yge4@lsu.edu}

\date{\today}

\begin{document}

\begin{abstract}
In this paper, we establish an analogue of the K\H{o}vari--S\'os--Tur\'an Theorem for $\GF(q)$-representable matroids.  For $2\leq s\leq t$, we show that if $M$ is a rank-$n$ simple $\GF(q)$-representable matroid having no $M(K_{s,t})$-restriction, then
\[
    |E(M)|=O_{q,s,t}\bigl(q^{(1-1/s)n}\bigr).
\]
In particular, we prove that the maximum number of elements in a simple rank-$n$ binary matroid with no $M(K_{2,t})$-restriction is $\Theta_{t}(2^{n/2})$ where the lower bound is obtained using binary Sidon sets.
\end{abstract}

\maketitle

\section{Introduction}

The {\it extremal number}, or {\it Tur\'an number}, $\operatorname{ex}(H;n)$ of a simple graph $H$ is the maximum number of edges in a simple $n$-vertex graph containing no subgraph isomorphic to $H$. The following is the celebrated theorem of Tur\'an~\cite{Turan41}.

\begin{theorem}[Tur\'an's Theorem]
Let $k\geq 2$. If $G$ is an $n$-vertex simple graph containing no subgraph isomorphic to $K_k$, then
\[
    |E(G)|\leq\left(\frac{k-2}{k-1}\right)\frac{n^2}{2}.
\]
Moreover, equality holds if and only if $(k-1)\mid n$ and $G$ is the complete $(k-1)$-partite graph whose parts all have size $n/(k-1)$.
\end{theorem}

Tur\'an-type problems have been a central topic in extremal graph theory.  For general graphs, Erd\H{o}s and Stone~\cite{Erdos_Stone} proved the following theorem.  We write $\chi(G)$ for the chromatic number of a graph $G$.

\begin{theorem}[Erd\H{o}s--Stone Theorem]
For every graph $G$ with at least one edge,
\[
    \operatorname{ex}(G;n)
      =\left(\frac{\chi(G)-2}{\chi(G)-1}+o(1)\right)\binom n2.
\]
\end{theorem}

The Erd\H{o}s--Stone Theorem gives an asymptotic value of $\operatorname{ex}(G;n)$ for every non-bipartite graph $G$.  For bipartite graphs, however, it only gives $\operatorname{ex}(G;n)=o(n^2)$, and much of the subsequent work on Tur\'an-type problems has focused on forbidding bipartite graphs.  One of the most famous results in this direction is the following upper bound of K\H{o}vari, S\'os, and Tur\'an~\cite{KST}.

\begin{theorem}[K\H{o}vari--S\'os--Tur\'an Theorem]
For positive integers $s$ and $t$ with $s\leq t$,
\[
    \operatorname{ex}(K_{s,t};n)=O_{s,t}(n^{2-1/s}).
\]
\end{theorem}

The main result of this paper is an analogue of the K\H{o}vari--S\'os--Tur\'an Theorem for simple $\GF(q)$-representable matroids.  Let $N$ be a nonempty simple $\GF(q)$-representable matroid.  The {\it extremal number}, or {\it Tur\'an number}, $\exq(N;n)$ is the maximum number of elements in a rank-$n$ simple $\GF(q)$-representable matroid with no $N$-restriction.  Before stating our result, we recall some Tur\'an-type results for representable matroids.  Our matroid notation follows~\cite{Oxley}.

For $\GF(q)$-representable matroids, projective geometries play a role analogous to that of complete graphs.  For $n\geq c\geq0$, the {\it Bose--Burton geometry}, denoted by $\BB_q(n,c)$, is obtained from $\PG(n-1,q)$ by deleting all points in a flat of rank $n-c$.  The following result of Bose and Burton~\cite{BB} is an analogue of Tur\'an's Theorem for projective geometries.

\begin{theorem}[Bose--Burton Theorem]
For $n\geq m\geq1$, if $M$ is a rank-$n$ simple $\GF(q)$-representable matroid with no $\PG(m-1,q)$-restriction, then
\[
    |E(M)|\leq\frac{q^n-q^{n-m+1}}{q-1}.
\]
Moreover, equality holds if and only if $M\cong\BB_q(n,m-1)$.
\end{theorem}

Geelen and Nelson~\cite{GeelenNelson} proved a matroidal analogue of the Erd\H{o}s--Stone Theorem.  For a simple rank-$r$ $\GF(q)$-representable matroid $N$, viewed as a restriction of $\PG(r-1,q)$, the {\it critical exponent} $c(N;q)$ is $r-r(F)$, where $F$ is a maximum-rank flat of $\PG(r-1,q)$ that is disjoint from $E(N)$.

\begin{theorem}[Matroidal Erd\H{o}s--Stone Theorem]
Let $q$ be a prime power, and let $N$ be a nonempty simple $\GF(q)$-representable matroid. Then
\[
    \exq(N;n)
      =\left(1-q^{1-c(N;q)}+o(1)\right)\frac{q^n-1}{q-1}.
\]
\end{theorem}

In the binary setting, Liu, Luo, Nelson, and Nomoto~\cite{LiuLuoNelsonNomoto} proved a stability refinement of the matroidal Erd\H{o}s--Stone Theorem. Like the Erd\H{o}s--Stone Theorem, the matroidal Erd\H{o}s--Stone Theorem is less informative when $c(N;q)=1$, that is, when $N$ is affine.  Although there has been extensive work on the bipartite Tur\'an problem in graph theory, the corresponding sparse problem for matroids has received very little attention.  To our knowledge, the only prior result directly addressing extremal numbers for excluded affine restrictions is the following theorem of Bonin and Qin~\cite[Lemma~21]{Bonin_Qin}.

\begin{theorem}
For $n\geq m\geq3$,
\[
    \operatorname{ex}_2(\AG(m-1,2);n)<2^{t_mn+1},
\]
where $t_m=1-2^{-(m-2)}$.
\end{theorem}

Bonin and Qin also gave a lower bound for $\operatorname{ex}_2(\AG(2,2);n)$~\cite[Corollary~22]{Bonin_Qin}.  It was pointed out by Geelen~\cite{MU_HJ} that the lower bound can be improved to give $\operatorname{ex}_2(\AG(2,2);n)=\Theta(2^{n/2})$.

The cycle matroids of complete bipartite graphs are affine over $\GF(q)$ for every prime power $q$.  Thus studying the extremal number of $M(K_{s,t})$ is a natural matroidal counterpart of the classical bipartite Tur\'an problem, and addresses a sparse case not resolved by the matroidal Erd\H{o}s--Stone Theorem.  Our main result, stated below, gives a K\H{o}vari--S\'os--Tur\'an-type upper bound for the extremal number of $M(K_{s,t})$.  Note that $M(K_{1,t})$ is the free matroid $U_{t,t}$ so it is contained as a restriction in every simple matroid of rank at least $t$.  We therefore focus on the cases $2\leq s\leq t$.

\begin{theorem}[Matroidal K\H{o}vari--S\'os--Tur\'an Theorem]\label{thm:main-upper}
For a prime power $q$ and integers $s$ and $t$ with $2\leq s\leq t$,
\[
    \exq(M(K_{s,t});n)=O_{q,s,t}\bigl(q^{(1-1/s)n}\bigr).
\]
\end{theorem}

A more detailed statement of the theorem along with its proof are given in Section~\ref{sec:upper}.  As in the bipartite Tur\'an problem, finding a matching lower bound remains open in general.  In the case $q=s=t=2$, determining $\operatorname{ex}_2(M(K_{2,2});n)$ is essentially equivalent to determining the maximum size of a binary Sidon set in $\F_2^n$.  In Section~\ref{sec:sidon}, we explain the intrinsic relation between these two problems and use results on Sidon sets to prove the following result.

\begin{theorem}\label{thm:binary-sidon}
For $t\geq2$ and $n\geq6$,
\[
    \operatorname{ex}_2(M(K_{2,t});n)=\Theta_t(2^{n/2}).
\]
\end{theorem}

\section{\texorpdfstring{$K_{s,t}$-restrictions in $\GF(q)$-representable matroids}{K(s,t)-restrictions in GF(q)-representable matroids}}\label{sec:representation}

Throughout the paper, we fix a prime power $q$, and let $\F$ be a field of order $q$.  We view a simple rank-$n$ $\GF(q)$-representable matroid $M$ as the vector matroid on a set $A\subseteq\F^n$. Since $M$ is simple, every vector in $A$ is nonzero, and no two vectors in $A$ are scalar multiples of one another.

The {\it $q$-expansion} of $A$ is the set $\qexp(A)=\{\lambda a:a\in A,\ \lambda\in\F-\{0\}\}$. The scalar orbits of distinct vectors in $A$ are disjoint, so $|\qexp(A)|=(q-1)|A|$. Let $M[\qexp(A)]$ denote the vector matroid with ground set $\qexp(A)$. This matroid is obtained from $M$ by replacing each element with a parallel class of size $q-1$. We omit the elementary proof of the following proposition.

\begin{proposition}\label{prop:simple_restriction_in_expansion}
Let $M$ be a simple rank-$n$ $\GF(q)$-representable matroid with ground set $A$. If a simple matroid $N$ is a restriction of $M[\qexp(A)]$, then $M$ has a restriction isomorphic to $N$.
\end{proposition}

The following lemma gives a sufficient condition for the existence of an $M(K_{s,t})$-restriction in terms of the $q$-expansion.

\begin{lemma}\label{lem:grid}
Let $M$ be a simple rank-$n$ $\GF(q)$-representable matroid with ground set $A$, and let $B=\qexp(A)$. For $2\leq s\leq t$, if there is a linearly independent subset
\[
\{p,v_2,v_3,\ldots,v_s,u_2,u_3,\ldots,u_t\}
\]
of $\F^n$ such that $p+v_i+u_j\in B$ for all $1\leq i\leq s$ and $1\leq j\leq t$, where $v_1=u_1=0$, then $M$ has an $M(K_{s,t})$-restriction.
\end{lemma}

\begin{proof}
Let $(\{x_1,x_2,\ldots,x_s\},\{y_1,y_2,\ldots,y_t\})$ be the bipartition of a copy $H$ of $K_{s,t}$. Let $G$ be obtained from $H$ by adding the edges
\[
\{x_ix_1:2\leq i\leq s\}\cup\{y_1y_j:2\leq j\leq t\},
\]
and let $T$ be the spanning tree of $G$, shown in Figure~\ref{fig:broom}, whose edge set is
\[
\{x_1y_1\}\cup\{x_ix_1:2\leq i\leq s\}\cup\{y_1y_j:2\leq j\leq t\}.
\]

\begin{figure}[htb]
\centering
\resizebox{6cm}{!}{\tikzset{every picture/.style={line width=0.75pt}}

\begin{tikzpicture}[x=0.75pt,y=0.75pt,yscale=-1,xscale=1]
\draw[fill=black] (216.07,136.94) .. controls (216.07,135.04) and
  (217.61,133.5) .. (219.52,133.5) .. controls (221.42,133.5) and
  (222.96,135.04) .. (222.96,136.94) .. controls (222.96,138.85) and
  (221.42,140.39) .. (219.52,140.39) .. controls (217.61,140.39) and
  (216.07,138.85) .. (216.07,136.94) -- cycle;
\draw[fill=black] (267.03,136.94) .. controls (267.03,135.04) and
  (268.58,133.5) .. (270.48,133.5) .. controls (272.38,133.5) and
  (273.92,135.04) .. (273.92,136.94) .. controls (273.92,138.85) and
  (272.38,140.39) .. (270.48,140.39) .. controls (268.58,140.39) and
  (267.03,138.85) .. (267.03,136.94) -- cycle;
\draw[fill=black] (163,102.44) circle (3.45);
\draw[fill=black] (163,152.82) circle (3.45);
\draw[fill=black] (163,171.44) circle (3.45);
\draw[fill=black] (327,92.94) circle (3.45);
\draw[fill=black] (327,143.5) circle (3.45);
\draw[fill=black] (327,162.22) circle (3.45);
\draw[fill=black] (327,180.94) circle (3.45);

\draw (219.52,136.94) -- (270.48,136.94);
\draw (163,102.44) -- (219.52,136.94);
\draw (163,152.82) -- (219.52,136.94);
\draw (159.55,171.44) -- (219.52,136.94);
\draw (270.48,136.94) -- (327,92.94);
\draw (270.48,136.94) -- (327,143.5);
\draw (270.48,136.94) -- (327,162.22);
\draw (270.48,136.94) -- (327,180.94);

\draw (160,115) node[anchor=north west,inner sep=0.75pt] {$\vdots$};
\draw (323,108) node[anchor=north west,inner sep=0.75pt] {$\vdots$};
\draw (215,116.4) node[anchor=north west,inner sep=0.75pt] {$x_1$};
\draw (258,116.4) node[anchor=north west,inner sep=0.75pt] {$y_1$};
\draw (138.5,97.4) node[anchor=north west,inner sep=0.75pt] {$x_2$};
\draw (131.5,145.4) node[anchor=north west,inner sep=0.75pt] {$x_{s-1}$};
\draw (138.5,166.4) node[anchor=north west,inner sep=0.75pt] {$x_s$};
\draw (338,87.4) node[anchor=north west,inner sep=0.75pt] {$y_2$};
\draw (339,137.4) node[anchor=north west,inner sep=0.75pt] {$y_{t-2}$};
\draw (338,158.4) node[anchor=north west,inner sep=0.75pt] {$y_{t-1}$};
\draw (337,180.4) node[anchor=north west,inner sep=0.75pt] {$y_t$};
\end{tikzpicture}}%
\caption{The spanning tree $T$ in $G$.}
\label{fig:broom}
\end{figure}
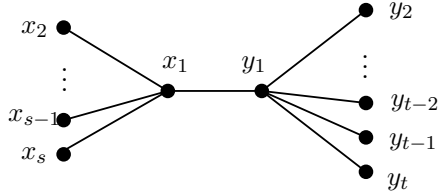

Define a map $\varphi:E(G)\to\F^n$ by
\[
\varphi(e)=
\begin{cases}
v_i, & \text{if }e=x_ix_1\text{ for some }2\leq i\leq s,\\
u_j, & \text{if }e=y_1y_j\text{ for some }2\leq j\leq t,\\
p+v_i+u_j, & \text{if }e=x_iy_j\text{ for some }1\leq i\leq s
\text{ and }1\leq j\leq t.
\end{cases}
\]
Since $\{p,v_2,v_3,\ldots,v_s,u_2,u_3,\ldots,u_t\}$ is linearly independent, $\varphi(E(T))$ is linearly independent. For each edge $x_iy_j$ of $E(H)-\{x_1y_1\}$, the vector $p+v_i+u_j$ is the sum of the vectors assigned to the edges of the unique path in $T$ from $x_i$ to $y_j$. Hence $\varphi(E(T))$ is a basis of the vector matroid on $\varphi(E(G))$. Moreover, the fundamental circuit of $\varphi(x_iy_j)$ with respect to $\varphi(E(T))$ is the image under $\varphi$ of the fundamental circuit of $x_iy_j$ with respect to $E(T)$. It follows, by comparing the fundamental circuits, that $\varphi:E(G)\to\varphi(E(G))$ is an isomorphism from the cycle matroid $M(G)$ to the vector matroid on
\[
\{v_2,v_3,\ldots,v_s,u_2,u_3,\ldots,u_t\}
\cup
\{p+v_i+u_j:1\leq i\leq s,\ 1\leq j\leq t\}.
\]

Since
\[
G\backslash\bigl(E(T)-\{x_1y_1\}\bigr)=H,
\]
it follows that the vector matroid on
\[
\{p+v_i+u_j:1\leq i\leq s,\ 1\leq j\leq t\}
\]
is isomorphic to $M(H)$. The last set is contained in $B$, so $M[\qexp(A)]$ has an $M(K_{s,t})$-restriction. By Proposition~\ref{prop:simple_restriction_in_expansion}, $M$ has an $M(K_{s,t})$-restriction.
\end{proof}

An ordered tuple $(x_1,x_2,\ldots,x_s)$ in $(\F^n)^s$ is {\it affinely independent} if $\{x_2-x_1,x_3-x_1,\ldots,x_s-x_1\}$ is linearly independent. Equivalently, there are no scalars $\alpha_1,\alpha_2,\ldots,\alpha_s$ in $\F$, not all zero, such that $\sum_{i=1}^s\alpha_i=0$ and $\sum_{i=1}^s\alpha_ix_i=0$. The ordered tuple is {\it affinely dependent} if it is not affinely independent.

Let $B$ be a subset of $\F^n$, and let $X=(x_1,x_2,\ldots,x_s)$ be affinely independent. A vector $y\in\F^n$ is a {\it common neighbor} of $X$ in $B$ if $x_i+y\in B$ for every $1\leq i\leq s$. We write $N_B(X)=\{y\in\F^n:x_i+y\in B\text{ for every }1\leq i\leq s\}$. Let $W_X=\spn\{x_i-x_1:2\leq i\leq s\}$. Then $\dim W_X=s-1$.

\begin{lemma}\label{lem:quotient-neighborhood}
For $2\leq s\leq t$, let $M$ be a simple rank-$n$ $\GF(q)$-representable matroid with ground set $A$, let $B=\qexp(A)$, and let $X$ be an affinely independent $s$-tuple $(x_1,x_2,\ldots,x_s)$ in $(\F^n)^s$. If the image of $x_1+N_B(X)$ in $\F^n/W_X$ contains $t$ linearly independent vectors, then $M$ has an $M(K_{s,t})$-restriction.
\end{lemma}

\begin{proof}
Let $\{y_1,y_2,\ldots,y_t\}$ be a subset of $N_B(X)$ so that the cosets
\[
x_1+y_1+W_X,x_1+y_2+W_X,\ldots,x_1+y_t+W_X
\]
are linearly independent in $\F^n/W_X$. Define
\begin{align*}
p&=x_1+y_1,\\
v_i&=x_i-x_1\text{ for }1\leq i\leq s,\text{ and}\\
u_j&=y_j-y_1\text{ for }1\leq j\leq t.
\end{align*}

First, we show that

\begin{sublemma}\label{sublem:linearly_independent}
$\{p,v_2,v_3,\ldots,v_s,u_2,u_3,\ldots,u_t\}$ is linearly independent.
\end{sublemma}

Suppose that there is a nontrivial linear relation
\begin{equation}\label{equ:first_le}
\alpha p+\sum_{i=2}^s\beta_iv_i+\sum_{j=2}^t\gamma_ju_j=0.
\end{equation}
Modulo $W_X$, this gives
\[
\alpha(x_1+y_1+W_X)+\sum_{j=2}^t\gamma_j(y_j-y_1+W_X)=0.
\]
Since $y_j-y_1+W_X=(x_1+y_j+W_X)-(x_1+y_1+W_X)$ for each $2\leq j\leq t$, we obtain
\[
\left(\alpha-\sum_{j=2}^t\gamma_j\right)(x_1+y_1+W_X)
+\sum_{j=2}^t\gamma_j(x_1+y_j+W_X)=0.
\]
The cosets $x_1+y_1+W_X,x_1+y_2+W_X,\ldots,x_1+y_t+W_X$ are linearly independent, so every $\gamma_j$ is zero and then $\alpha=0$. Therefore, \eqref{equ:first_le} becomes
\[
\sum_{i=2}^s\beta_iv_i=0.
\]
Since $X$ is affinely independent, $v_2,v_3,\ldots,v_s$ are linearly independent. It follows that every $\beta_i$ is also zero, contradicting the assumption that~\eqref{equ:first_le} is nontrivial. Therefore~\ref{sublem:linearly_independent} holds.

Finally, $p+v_i+u_j=x_i+y_j\in B$ for all $1\leq i\leq s$ and $1\leq j\leq t$. Since $v_1=u_1=0$, the result follows from Lemma~\ref{lem:grid}.
\end{proof}

\begin{corollary}\label{cor:common-bound}
For $2\leq s\leq t$, let $M$ be a simple rank-$n$ $\GF(q)$-representable matroid with ground set $A\subseteq\F^n-\{0\}$, and let $B=\qexp(A)$. If $M$ has no $M(K_{s,t})$-restriction, then every affinely independent $X\in(\F^n)^s$ satisfies
\[
|N_B(X)|\leq q^{s+t-2}.
\]
\end{corollary}

\begin{proof}
By Lemma~\ref{lem:quotient-neighborhood}, the image of $x_1+N_B(X)$ in $\F^n/W_X$ spans a subspace of dimension at most $t-1$, so it has at most $q^{t-1}$ elements. Each coset of $W_X$ has $q^{s-1}$ elements. Therefore
\[
|N_B(X)|\leq q^{t-1}q^{s-1}=q^{s+t-2}.\qedhere
\]
\end{proof}

This corollary is analogous to the basic observation that, in a graph for which $K_{s,t}$ does not occur as a subgraph, no set of $s$ vertices has more than $t-1$ common neighbors.

\section{\texorpdfstring{Number of affinely independent $s$-tuples}{Number of affinely independent s-tuples}}\label{sec:affine-tuples}

Recall that an $s$-tuple $(x_1,x_2,\ldots,x_s)\in(\F^n)^s$ is affinely dependent if there is a nonzero vector $(\alpha_1,\alpha_2,\dots,\alpha_s)\in\F^s$ such that $\sum_{i=1}^s\alpha_i=0$ and $\sum_{i=1}^s\alpha_ix_i=0$. We call such a vector a {\it solution} for $(x_1,x_2,\ldots,x_s)$. A solution is {\it normalized} if its first nonzero entry is one. Since every solution can be multiplied by a nonzero scalar to obtain a normalized solution, an $s$-tuple is affinely dependent if and only if it has a normalized solution. Let $\mathcal L_s$ be the set of nonzero vectors $(\alpha_1,\alpha_2,\ldots,\alpha_s)\in\F^s$ such that $\sum_i\alpha_i=0$ and the first nonzero entry is one. 

\begin{proposition}
$|\mathcal L_s|=\dfrac{q^{s-1}-1}{q-1}$.
\end{proposition}

\begin{proof}
Let $\mathcal L_s'$ be the set of nonzero vectors $(\alpha_1,\alpha_2,\ldots,\alpha_s)\in\F^s$ such that $\sum_i\alpha_i=0$. Since the first $s-1$ entries determine the last entry, there are $q^{s-1}$ such vectors, one of which is zero. Therefore,
\[
|\mathcal L_s'|=q^{s-1}-1.
\]
Clearly
\[
\mathcal L_s'=\{\lambda v:\lambda\in\F-\{0\},\text{ and } v\in\mathcal L_s\}.
\]
If $\lambda_1 v=\lambda_2 w$ for $\lambda_1,\lambda_2\in\F-\{0\}$ and $v,w\in\mathcal L_s$, then comparing the first nonzero entries gives $\lambda_1=\lambda_2$. Hence $v=w$. Therefore,
\[
|\mathcal L_s'|=(q-1)|\mathcal L_s|.
\]
It follows that
\[
|\mathcal L_s|=\frac{|\mathcal L_s'|}{q-1}
=\frac{q^{s-1}-1}{q-1}.
\]
\end{proof}

For convenience, write
\[
    d_{q,s}=\frac{q^{s-1}-1}{q-1}.
\]

\begin{lemma}\label{lem:dependent-tuples}
Let $B$ be a subset of $\F^n$.  The number of affinely dependent $s$-tuples in $B^s$ is at most
 $d_{q,s}|B|^{s-1}$.
\end{lemma}

\begin{proof}
For $\alpha=(\alpha_1,\alpha_2,\ldots,\alpha_s)\in\mathcal L_s$, let $\mathcal D_B(\alpha)$ be the set of $s$-tuples $(b_1,b_2,\ldots,b_s)$ in $B^s$ such that
\begin{align}
    \sum_{i=1}^s\alpha_i b_i=0.\label{equ:second_equation}
\end{align}

Choose $i_0$ in $\{1,2,\ldots,s\}$ with $\alpha_{i_0}\neq0$.  Suppose $(b_1,b_2,\ldots,b_s)\in \mathcal D_B(\alpha)$. Once the entries $b_i$ with $i\neq i_0$ are chosen, the equation~\eqref{equ:second_equation} determines
\[
    b_{i_0}=-\alpha_{i_0}^{-1}\sum_{i\neq i_0}\alpha_i b_i.
\]
This vector may not lie in $B$, nevertheless, $|\mathcal D_B(\alpha)|\leq |B|^{s-1}$.

Since every affinely dependent $s$-tuple has a normalized solution, every such $s$-tuple belongs to $\mathcal D_B(\alpha)$ for some $\alpha\in\mathcal L_s$.  Therefore, the number of affinely dependent $s$-tuples in $B^s$ is at most
\[
    \sum_{\alpha\in\mathcal L_s}|\mathcal D_B(\alpha)|
      \leq d_{q,s}|B|^{s-1}.
\]
\end{proof}

The following result is an immediate consequence of the last lemma.

\begin{corollary}\label{cor:independent-tuples}
Let $B$ be a subset of $\F^n$.  The number of affinely independent $s$-tuples in $B^s$ is at least
\[
    \max\{0,(|B|-d_{q,s})|B|^{s-1}\}.
\]
\end{corollary}

\section{The proof of the upper bound}\label{sec:upper}

In this section, we prove the following strengthening of Theorem~\ref{thm:main-upper}.

\begin{theorem}\label{thm:detailed-upper}
For a prime power $q$ and integers $s$ and $t$ with $2\leq s\leq t$,
\[
    \exq(M(K_{s,t});n)
    \leq\frac{1}{q-1}\max\left\{
       2d_{q,s},
       (2q^{s+t-2})^{1/s}q^{(1-1/s)n}
    \right\}.
\]
\end{theorem}

\begin{proof}
Let $M$ be a rank-$n$ simple $\GF(q)$-representable matroid with ground set $A$ and with no $M(K_{s,t})$-restriction.  Let $B=\qexp(A)$. Since $M$ is simple, $|B|=(q-1)|E(M)|$. If $|B|\leq 2d_{q,s}$, then $|E(M)|\leq \frac{2d_{q,s}}{q-1}$ and the conclusion is immediate. So suppose that $|B|>2d_{q,s}$.  We double count the pairs $(X,y)$ such that $X$ is an affinely independent $s$-tuple $(x_1,x_2,\ldots,x_s)$ in $(\F^n)^s$ and
\[
    y\in N_B(X).
\]
Let $m$ be the number of such pairs.  Let $\mathcal X$ be the collection of affinely independent $s$-tuples in $(\F^n)^s$. Then
\begin{align}
    m=\sum_{X\in\mathcal X}|N_B(X)|.\label{equ:RHS}
\end{align}

On the other hand, let $\mathcal I$ be the collection of affinely independent $s$-tuples in $B^s$. Since translation preserves affine independence, the map
\[
    (I,y)\longmapsto(I-y,y)
\]
is a bijection from $\mathcal I\times\F^n$ to the set of pairs counted by $m$, where
\[
    I-y=(b_1-y,b_2-y,\ldots,b_s-y)
\]
when $I=(b_1,b_2,\ldots,b_s)$. Moreover, $y\in N_B(I-y)$ because $(b_i-y)+y=b_i\in B$ for each $i$.  Conversely, if $(X,y)$ is a pair counted by $m$, then $X+y\in\mathcal I$.  Therefore,
\begin{align}
    m=q^n|\mathcal I|.\label{equ:LHS}
\end{align}
Combining equations~\eqref{equ:RHS} and~\eqref{equ:LHS}, we obtain
\[
    q^n|\mathcal I|=\sum_{X\in\mathcal X}|N_B(X)|.
\]
By Corollary~\ref{cor:common-bound}, we have $|N_B(X)|\leq q^{s+t-2}$ for every $X\in\mathcal X$.  Hence
\[
    q^n|\mathcal I|\leq q^{s+t-2}|\mathcal X|.
\]
Using the trivial bound $|\mathcal{X}|\leq (q^n)^s$, we get $q^n|\mathcal I|\leq q^{s+t-2}q^{sn}$, so
\[
    |\mathcal I|\leq q^{s+t-2}q^{(s-1)n}.
\]
By Corollary~\ref{cor:independent-tuples},
\[
    (|B|-d_{q,s})|B|^{s-1}
      \leq q^{s+t-2}q^{(s-1)n}.
\]
Since, by assumption, $|B|>2d_{q,s}$, we have
\[
    \frac{|B|^s}{2}
      \leq q^{s+t-2}q^{(s-1)n}.
\]
Thus
\[
    |B|\leq(2q^{s+t-2})^{1/s}q^{(1-1/s)n}.
\]
Finally, $|B|=(q-1)|E(M)|$, and hence Theorem~\ref{thm:detailed-upper} holds.
\end{proof}

Since $(2d_{q,s})/(q-1)$ and $(2q^{s+t-2})^{1/s}/(q-1)$ depend only on $q,s,t$, we conclude that
\[
    \exq(M(K_{s,t});n)=O_{q,s,t}\bigl(q^{(1-1/s)n}\bigr),
\]
as stated in Theorem~\ref{thm:main-upper}.

\section{Lower bound and binary Sidon sets}\label{sec:sidon}

Constructing lower bounds that match Theorem~\ref{thm:detailed-upper} in full generality appears difficult. In this section, we focus on the special case when $q=s=2$.  In number theory, a {\it Sidon sequence} is a sequence $a_1<a_2<\cdots$ of positive integers such that all sums $a_i+a_j$ with $i\leq j$ are distinct.  This notion was generalized to arbitrary groups by Babai and S\'os~\cite{BabaiSos}, who called the resulting objects Sidon sets.  A {\it binary Sidon set} is a subset $S$ of $\F_2^n$ such that, for all distinct $a,b,c,d\in S$, we have
\[
    a+b\neq c+d.
\]
Equivalently, no four distinct elements of $S$ sum to zero.

\begin{proposition}\label{prop:sidon-equivalence}
Let $S\subseteq\F_2^n-\{0\}$.  The vector matroid on $S$ has no $M(K_{2,2})$-restriction if and only if $S$ is a binary Sidon set.
\end{proposition}

\begin{proof}
Since $M(K_{2,2})=M(C_4)$, an $M(K_{2,2})$-restriction is a four-element circuit.  Four distinct nonzero vectors in $\F_2^n$ form such a circuit if and only if their sum is zero.
\end{proof}

The condition $0\notin S$ is harmless for our purposes.  If a binary Sidon set $S$ contains zero, then, for every $b\in\F_2^n-S$, the translate $S+b$ is a binary Sidon set that avoids zero.  Moreover, if $S$ does not span $\F_2^n$, the vector matroid on $S$ may be replaced by its direct sum with $n-r(S)$ copies of $U_{1,1}$. The resulting matroid is a rank-$n$ simple binary matroid with at least $|S|$ elements and no $M(K_{2,2})$-restriction. Therefore, determining $\operatorname{ex}_2(M(K_{2,2});n)$ is essentially equivalent to determining the maximum size of a binary Sidon set in $\F_2^n$. The following upper bound and constructions are recorded in~\cite[Proposition~1.2 and Remark~1.3]{CzerwinskiPott}.

\begin{proposition}\label{prop:sidon-upper}
If $n\geq6$ and $S$ is a binary Sidon set in $\F_2^n$, then
\[
|S|\leq
\begin{cases}
2^{(n+1)/2}-2, & \text{if $n$ is odd},\\
\lfloor2^{(n+1)/2}+0.5\rfloor, & \text{if $n$ is even}.
\end{cases}
\]
\end{proposition}

\begin{proposition}\label{prop:sidon-lower}
For $n\geq6$, there are binary Sidon sets in $\F_2^n$ of size
\[
\begin{cases}
2^k+2, & \text{if $n=2k$ and $k$ is even},\\
2^k+1, & \text{if $n=2k$ and $k$ is odd},\\
2^{k-1}+2^{k/2}, & \text{if $n=2k-1$ and $k$ is even},\\
2^{k-1}+2^{(k-1)/2}, & \text{if $n=2k-1$ and $k$ is odd}.
\end{cases}
\]
\end{proposition}

Using Propositions~\ref{prop:sidon-upper} and~\ref{prop:sidon-lower}, we obtain the following.

\begin{corollary}\label{cor:k22}
For $n\geq6$,
\[
    2^{(n-1)/2}+1
      \leq\operatorname{ex}_2(M(K_{2,2});n)
      \leq2^{(n+1)/2}+0.5.
\]
\end{corollary}

For $t\geq2$, every $M(K_{2,t})$-restriction contains an $M(K_{2,2})$-restriction.  Hence
\[
    \operatorname{ex}_2(M(K_{2,2});n)
      \leq\operatorname{ex}_2(M(K_{2,t});n).
\]
Combining Theorem~\ref{thm:detailed-upper} with Corollary~\ref{cor:k22}, we conclude with the following more precise restatement of Theorem~\ref{thm:binary-sidon}.
\begin{corollary}
For $t\geq 2$ and $n\geq6$,
\[
    2^{(n-1)/2}+1
      \leq\operatorname{ex}_2(M(K_{2,t});n)
      \leq2^{(t+1)/2}2^{n/2}.
\]
\end{corollary}

\section*{Acknowledgments}

The author thanks James Oxley for many helpful discussions and comments.


\begin{thebibliography}{99}

\bibitem{BabaiSos}
L.~Babai and V.~T.~S\'os,
Sidon sets in groups and induced subgraphs of Cayley graphs,
{\it European J. Combin.} \textbf{6}~(1985), 101--114.

\bibitem{Bonin_Qin}
J.~Bonin and H.~Qin,
Size functions of subgeometry-closed classes of representable combinatorial geometries,
{\it Discrete Math.} \textbf{224}~(2000), 37--60.

\bibitem{BB}
R.~C.~Bose and R.~C.~Burton,
A characterization of flat spaces in a finite geometry and the uniqueness of the Hamming and the MacDonald codes,
{\it J. Combin. Theory} \textbf{1}~(1966), 96--104.

\bibitem{CzerwinskiPott}
I.~Czerwinski and A.~Pott,
On large Sidon sets,
{\it J. Combin. Theory Ser. A} \textbf{220}~(2026), Paper 106129.

\bibitem{Erdos_Stone}
P.~Erd\H{o}s and A.~H.~Stone,
On the structure of linear graphs,
{\it Bull. Amer. Math. Soc.} \textbf{52}~(1946), 1087--1091.

\bibitem{GeelenNelson}
J.~Geelen and P.~Nelson,
An analogue of the Erd\H{o}s--Stone theorem for finite geometries,
{\it Combinatorica} \textbf{35}~(2015), 209--214.

\bibitem{MU_HJ}
J.~Geelen,
The Geometric Density Hales--Jewett Theorem,
\url{https://matroidunion.org/?p=601}.

\bibitem{KST}
T.~K\H{o}vari, V.~T.~S\'os, and P.~Tur\'an,
On a problem of K.~Zarankiewicz,
{\it Colloq. Math.} \textbf{3}~(1954), 50--57.

\bibitem{LiuLuoNelsonNomoto}
H. Liu, S. Luo, P. Nelson, and K. Nomoto,
Stability and exact Turán numbers for matroids,
{\it J. Combin. Theory Ser. B} \textbf{143}~(2020), 29--41.

\bibitem{Oxley}
J.~Oxley, 
{\it Matroid Theory}, Second edition, 
Oxford University Press, New York, 2011.

\bibitem{Turan41}
P.~Tur\'an,
Eine Extremalaufgabe aus der Graphentheorie,
{\it Mat. Fiz. Lapok} \textbf{48}~(1941), 436--452.

\end{thebibliography}
\end{document}